\documentclass{scrartcl}
\usepackage[T1]{fontenc}
\usepackage[english]{babel}
\usepackage[latin1]{inputenc}
\usepackage{amsmath}
\usepackage{amsthm,amssymb}
\usepackage{csquotes}
\usepackage{amssymb}
\usepackage{graphicx}
\usepackage{color, pgf}
\usepackage{hyperref}
\hypersetup{colorlinks=true, urlcolor=blue, citecolor=red} 
\newtheorem{thm}{Theorem}[section]
\newtheorem{defn}[thm]{Definition}
\newtheorem{lem}[thm]{Lemma}
\newtheorem{cor}[thm]{Corollary}
\newtheorem{rem}[thm]{Remark}
\newtheorem{prop}[thm]{Proposition}
\newtheorem{theorem}{Theorem}[section]

\newtheorem{lemma}[theorem]{Lemma}

\begin{document}

\title{Renormalized solutions for stochastic $p$-Laplace equations with $L^1$-initial data: The multiplicative case}
\author{Niklas Sapountzoglou \thanks{Faculty of Mathematics, University of Duisburg-Essen, Thea-Leymann-Str.9, 45127 Essen, Germany \href{mailto:niki_sapo@msn.com}{\texttt{niki\_sapo@msn.com}}}\and  Aleksandra Zimmermann \thanks{Faculty of Mathematics, University of Duisburg-Essen, Thea-Leymann-Str.9, 45127 Essen, Germany \href{mailto:aleksandra.zimmermann@uni-due.de}{\texttt{aleksandra.zimmermann@uni-due.de}}}}
\date{}

\maketitle

\begin{abstract}
We consider a $p$-Laplace evolution problem with multiplicative noise on a bounded domain $D\subset\mathbb{R}^d$ with homogeneous Dirichlet boundary conditions for $1<p<\infty$. The random initial data is merely integrable. Consequently, the key estimates are available with respect to truncations of the solution. We introduce the notion of renormalized solutions for multiplicative stochastic $p$-Laplace equations with $L^1$-initial data and study existence and uniqueness of solutions in this framework.
\end{abstract}

\section{Introduction}
\subsection{Motivation of the study}
In the theory of stochastic partial differential equations (SPDEs), square integrability of the initial data is a rather technical assumption. For square integrable initial data $u_0$, the stochastic $p$-Laplace evolution problems can be solved with classical methods for nonlinear, monotone SPDEs (see, e.g. \cite{EP}, \cite{GDPJZ}, \cite{WLMR} and \cite{Br} for systems). In applications, one often has flawed or irregular data and therefore it may be reasonable to study more general, merely integrable random initial conditions. From the results of \cite{DB}, \cite{BM97} and \cite{DBFMHR} it is well known that the deterministic $p$-Laplace evolution equation with $L^1$-data is not well-posed in the variational setting for $1<p<d$, where $d\in\mathbb{N}$ is the space dimension. In this case the equation can be addressed within the framework of renormalized solutions. The notion of renormalization summarizes different strategies to get rid of infinities (see \cite{DeLM}) that may appear in physical models. It has been introduced to partial differential equations by Di Perna and Lions in the study of Boltzmann equation (see \cite{DL}) and then extended to many elliptic and parabolic problems (see, e.g., \cite{EF}, \cite{PBLBTGRGMPJLV}, \cite{BM97}, \cite{BR} and the references therein). The main idea is to develop an equation that is satisfied for $v=S(u)$, where the nonlinear function $S$ is chosen in order to remove infinite quantities of the solution $u$. This strategy has been applied for stochastic transport equations in \cite{AF}, \cite{CO} and for the Boltzmann equation with stochastic kinetic transport in \cite{PS}.\\

In the study of singular SPDEs such as the Kardar-Parisi-Zhang (KPZ) equation, the idea of renormalization has recently been revisited (see \cite{GIP2015}, \cite{Hai2014} and the references therein). For this type of equations, renormalized solutions are obtained as limits of solutions to regularized problems with addition of diverging correction terms. These terms arise from a renormalization group which is defined in terms of an associated regularity structure.\\
Thanks to the new techniques developed in the theories of regularity structures and rough paths a huge progess in the study of singular SPDEs has been achieved in the last decade.\\

However, the classical $L^1$-theory for nonlinear SPDEs of $p$-Laplace type is an independent topic which has been recently addressed in \cite{SZ20a}, where the notion of renormalized solutions in the sense of \cite{BM97} has extended to stochastic $p$-Laplace evolution equations. In a subsequent contribution (see, e.g., \cite{NSAZ}) existence and uniqueness of renormalized solutions to the stochastic $p$-Laplace evolution problem with random initial data in $L^1(\Omega\times D)$ has been shown in the case of an additive stochastic perturbation, i.e., with an It\^{o} integral $\int_0^t \Phi \,d\beta$ on the right-hand side of the equation, where $\Phi$ is a progressively measurable and square integrable stochastic process. In this contribution, we are interested in well-posedness of the stochastic $p$-Laplace equation in the  multiplicative case, i.e., for an It{\^{o}} integral $\int_0^t \Phi(u) \,d\beta$ on the right-hand side of the equation, where, roughly speaking, $\Phi(u)$ is a Lipschitz function of the solution $u$.\\
In the classical $L^2$-theory, with the solution of the additive problem and an appropriate contraction principle at hand, existence of solutions to the corresponding multiplicative problem can be written in a few lines applying a fixed point argument.\\
In the multiplicative case with $L^1$ initial data, the situation changes dramatically due to new phenomena. The main difficulty is the combination of $L^1$-spatial regularity of $u$ with additional terms entering the renormalized formulation from the It{\^{o}} correction and the non-cancellation of stochastic integrals in differences of solutions. In our study, we put the spotlight on the new techniques developed for the multiplicative case and refer to \cite{NSAZ} for known results in order to avoid doubling of arguments.  

\subsection{Formulation of the problem and assumptions}
Let $(\Omega, \mathcal{F}, P,(\mathcal{F}_t)_{t \in [0,T]}, (\beta_t)_{t\in [0,T]})$ be a stochastic basis with a countably generated \linebreak probability space $(\Omega, \mathcal{F}, P)$, a filtration $(\mathcal{F}_t)_{t \in [0,T]}\subset \mathcal{F}$ satisfying the usual assumptions and a real valued, $\mathcal{F}_t$-Brownian motion $(\beta_t)_{t\in [0,T]}$. Let  $D \subset \mathbb{R}^d$ be a bounded Lipschitz domain, $T>0$, $Q_T=(0,T) \times D$ and $p>1$. Furthermore, let $u_0: \Omega \to L^1(D)$ be $\mathcal{F}_0$-measurable and $\Phi: \Omega \times [0,T] \times \mathbb{R} \to \mathbb{R}$ a function satisfying the following properties:
\begin{enumerate}
\item[$(A1)$] $\Phi(\omega,t,0)=0$ for almost every $(\omega,t) \in \Omega \times [0,T]$,
\item[$(A2)$] $\lambda \mapsto \Phi(\omega, t, \lambda)$ is Lipschitz-continuous for almost every $(\omega,t) \in \Omega \times [0,T]$, i.e. there exists $L>0$ such that for all $\lambda_1, \lambda_2 \in \mathbb{R}$
\begin{align*}
|\Phi(\omega, t, \lambda_2) - \Phi(\omega, t, \lambda_1)| \leq L | \lambda_2 - \lambda_1|
\end{align*}
for almost every $(\omega,t) \in \Omega \times [0,T]$.
\item[$(A3)$] $(\omega, t) \mapsto \Phi(\omega, t , \lambda)$ is progressively measurable for all $\lambda \in \mathbb{R}$.
\end{enumerate}
We are interested in well-posedness to the following stochastic $p$-Laplace evolution problem
\begin{align}\label{1}
du - \textnormal{div}\,(|\nabla u|^{p-2} \nabla u)\,dt &= \Phi(u) ~d\beta ~~~&\textnormal{in}~ \Omega \times Q_T, \notag\\
u&=0 ~~~&\textnormal{on}~ \Omega \times (0,T) \times \partial D, \\
u(0, \cdot)&=u_0 ~~~&\in ~L^1(\Omega \times D). \notag
\end{align}

\subsection{Outline}
In Section \ref{S2}, we will briefly discuss the well-posedness of \eqref{1} for square integrable initial data. In Section \ref{S3} we will formulate and prove a contraction principle for strong solutions that will serve as a basis for the proof of the existence theorem (see Theorem \ref{Theorem 6.2}) in Section \ref{S6}. After some technical preliminaries in Section \ref{S4}, we formulate the notion of renormalized solutions for \eqref{1} in Section \ref{S5}. The uniqueness of renormalized solutions in formulated in Theorem \ref{Theorem 7.1} and contained in Section \ref{S7} together with its proof.
\section{Strong solutions}\label{S2}
\begin{thm}\label{Theorem 2.1}
Let $u_0 \in L^2(\Omega \times D)$ be $\mathcal{F}_0$-measurable. Then there exists a unique strong solution to (\ref{1}), i.e., an $\mathcal{F}_t$-adapted stochastic process $u:\Omega\times [0,T]\to L^2(D)$ such that $u \in L^p(\Omega; L^p(0,T; W_0^{1,p}(D))) \cap L^2(\Omega; \mathcal{C}([0,T];L^2(D)))$, $u(0, \cdot)=u_0$ in $L^2(\Omega \times D)$ and
\begin{align*}
u(t) - u_0 - \int_0^t \textnormal{div}\,(|\nabla u|^{p-2} \nabla u)\, ds = \int_0^t \Phi(u) \,d\beta
\end{align*}
in $W^{-1,p'}(D) + L^2(D)$ for all $t \in [0,T]$ and a.s. in $\Omega$.
\end{thm}
\begin{rem}\label{Remark 2.2}
A-priori, we do not know that the term $\int_0^t \textnormal{div}\,(|\nabla u|^{p-2} \nabla u) \, ds $ is an  element of $L^2(D)$ for all $t \in [0,T]$ and a.s. in $\Omega$, but since we know that all other terms in the equation of Theorem \ref{Theorem 2.1} are elements of $L^2(D)$ for all $t \in [0,T]$ and a.s. in $\Omega$ it follows that $\int_0^t \textnormal{div} (|\nabla u|^{p-2} \nabla u) ~ds  \in L^2(D)$ for all $t \in [0,T]$ and a.s. in $\Omega$. Therefore this equation is an equation in $L^2(D)$.
\end{rem}
\begin{proof}
Similar as in \cite{NSAZ}, the existence result is a consequence of \cite{NVKBLR}, Chapter II, Theorem 2.1 and Corollary 2.1.
The uniqueness result is a consequence of \cite{NVKBLR}, Chapter II, Theorem 3.1.
\end{proof}
\section{Contraction principle}\label{S3}
\begin{thm}\label{Theorem 3.1}
Let $u_0, v_0 \in L^2(\Omega \times D)$ and $u$ and $v$ strong solutions to the problem \eqref{1} with initial value $u_0$ and $v_0$, respectively. Then
\begin{align*}
\sup\limits_{t \in [0,T]} \mathbb{E} \int_D |u(t) - v(t)| \,dx \leq \mathbb{E} \int_D |u_0 - v_0| \,dx
\end{align*}
a.s. in $\Omega$.
\end{thm}
\begin{proof}
We subtract the equations for $u$ and $v$ and we get
\begin{align}\label{222}
u(t) - v(t) = u_0 - v_0 + \int_0^t \textnormal{div} (|\nabla u|^{p-2} \nabla u - |\nabla v|^{p-2} \nabla v) \,ds + \int_0^t \Phi(u) - \Phi(v) \, d\beta
\end{align}
for all $t \in [0,T]$ and a.s. in $\Omega$. Now, for every $\delta >0$, we apply the It\^{o} formula pointwise a.s. with respect to $x\in D$ with a coercive approximation of the absolute value $N_{\delta}$ in \eqref{222} which is defined as in Proposition 5 in \cite{VZ18} (see also \cite{NSAZ}) and we get
\begin{align*}
&\int_D N_{\delta}(u(t)-v(t))  \, dx = \int_D N_{\delta}(u_0 - v_0) \, dx \\
+ &\int_0^t \langle \operatorname{div}(|\nabla u|^{p-2} \nabla u - |\nabla v|^{p-2} \nabla v), N_{\delta}'(u-v) \rangle \,ds \\
+ &\int_0^t \int_D (\Phi(u) - \Phi(v)) N_{\delta}'(u-v) \, dx\, d\beta + \frac{1}{2} \int_0^t \int_D N_{\delta}''(u-v) (\Phi(u) - \Phi(v))^2 \, dx \, ds.
\end{align*}
Applying the expectation and discarding the nonpositive term coming from the $p$-Laplace yields
\begin{align*}
&\mathbb{E} \int_D N_{\delta}(u(t)-v(t))  \, dx \\
\leq \, &\mathbb{E} \int_D N_{\delta}(u_0 - v_0) \, dx + \frac{1}{2} \mathbb{E}\int_0^t \int_D N_{\delta}''(u-v) (\Phi(u) - \Phi(v))^2 \, dx \, ds \\
\leq \, &\mathbb{E} \int_D N_{\delta}(u_0 - v_0) \, dx + \frac{1}{\delta} \mathbb{E}\int_0^t \int_D \chi_{\{|u-v| \leq \delta\}} (\Phi(u) - \Phi(v))^2 \, dx \, ds \\
\leq \, &\mathbb{E} \int_D N_{\delta}(u_0 - v_0) \, dx + \frac{1}{\delta} \mathbb{E}\int_0^t \int_D \chi_{\{|u-v| \leq \delta\}} L^2 \delta^2 \, dx \, ds \\
\leq \, &\mathbb{E} \int_D N_{\delta}(u_0 - v_0) \, dx + \delta L^2 T |D|.
\end{align*}
Now, passing to the limit with $\delta\rightarrow 0^+$ yields
\begin{align*}
\mathbb{E} \int_D |u(t) - v(t)| \,dx \leq \mathbb{E} \int_D |u_0 - v_0| \,dx
\end{align*}
for all $t \in [0,T]$.
\end{proof}
\begin{rem}
In \cite{NSAZ}, the integrand of the stochastic integral in the equation for $u-v$ is equal to $0$ since we consider the additive case there. Therefore, in this case, the integrand of the stochastic integral in the equation for $\int_D N_{\delta}(u(t)-v(t))  \, dx$ is equal to $0$ as well. But in the multiplicative case the stochastic integral does not vanish and does not tend to $0$ uniformly in $t\in [0,T]$, a.s. in $\Omega$ for $\delta \to 0^+$. Hence, in contrast to \cite{NSAZ}, we have to apply the expectation before applying the supremum over $t \in [0,T]$ which leads to a weaker comparison principle as in \cite{NSAZ}.
\end{rem}
\section{Technical preliminaries}\label{S4}
\subsection{Sum and intersection spaces}
For $1<p<\infty$ let $p'$ bet the conjugate exponent to $p$. Let us recall that the spaces $L^1(D)$ and $W^{-1,p'}(D)$ are continuously embedded into the space of distributions $\mathcal{D}'(D)$. Moreover, by density of the test functions in $L^1(D)$ and in $L^{p'}(D)$, it follows that $L^1(D)\cap W^{-1,p'}(D)$ is dense in $L^1(D)$ and $W^{-1,p'}(D)$. Therefore, the sum space
\[L^1(D)+W^{-1,p'}(D)=\{w=u+v \ |\ u\in L^1(D), \ v\in W^{-1,p'}(D)\}\]
is well defined and a Banach space with respect to the norm
\[|\Vert w\Vert|:=\inf\{\Vert u\Vert_{L^1}+\Vert v\Vert_{W^{-1,p'}} \ | \ u\in L^1(D), \ v\in W^{-1,p'}(D), \ u+v=w\},\]
see, e.g., \cite{FKS05}, p.23. Moreover, the dual space is given by $(L^1(D)+W^{-1,p'}(D))'=W^{1,p}_0(D)\cap L^{\infty}(D)$.
\subsection{Weakly continuous functions}
For a Banach space $V$ with dual space $V'$, a function $u:[0,T]\rightarrow V$ is called weakly continuous, iff the function
\[[0,T]\ni t\mapsto \langle u(t),v\rangle_{V,V'} \]
is continuous for all $v\in V'$. The locally convex space of weakly continuous functions with values in $V$ will be denoted by $\mathcal{C}_w([0,T];V)$ in the sequel. For further details on the properties these spaces we refer to \cite{GGZ}, pp.120-126.
In particular, the following result holds true
\begin{lemma}[see \cite{Temam}, Lemma 1.4, p.263]\label{Lemma 1.4}
Let $X$ and $Y$ be Banach spaces such that $X\subset Y$ with a continuous injection. If a function $\phi$ belongs to $L^{\infty}(0,T;X)$ and is weakly continuous with values in $Y$, then $\phi$ is weakly continuous with values in $X$.
\end{lemma}
\subsection{A generalized It{\^{o}} formula}
\begin{prop}\label{Proposition 4.2}
Let $G \in L^{p'}(\Omega \times Q_T)^d$, $g\in L^2(\Omega\times Q_T)$, $f\in L^1(\Omega\times Q_T)$ be progressively measurable, $u_0 \in L^1(\Omega \times D)$ be $\mathcal{F}_0$-measurable. Define the continuous, $L^1(D)+W^{-1,p'}(D)$-valued process $u$ by the equality
\begin{align}\label{2}
u(t) - u_0 +\int_0^t (-\textnormal{div}\,G +f) \,ds = \int_0^t g \,d\beta
\end{align}
for all $t \in [0,T]$ and a.s. in $\Omega$. If for its $dP\otimes dt$-equivalence class we have
$u \in L^p(\Omega; L^p(0,T;W_0^{1,p}(D)))$ and $\operatorname{ess}\sup_{t\in [0,T]} \mathbb{E}\Vert u(t)\Vert_{L^1}<\infty$, then, for all $\psi \in C^{\infty}([0,T] \times \overline{D})$ and all $S\in W^{2,\infty}(\mathbb{R})$ with $S''$ piecewise continuous 
such that $S'(0)=0$ or $\psi(t,x) =0$ for all $(t,x) \in [0,T]\times \partial D$, we have
\begin{align}\label{Itoformulat}
&(S(u(t)),\psi(t))_2 - (S(u_0),\psi(0))_2+ \int_0^t \langle -\operatorname{div}\,G+f,S'(u)\psi\rangle\,ds\nonumber\\ 
&=\int_0^t (S'(u)g,\psi)_2\,d\beta + \int_0^t (S(u),\psi_t)_2 \,ds + \frac{1}{2} \int_0^t \int_D  S''(u)g^2 \psi  \,dx\,ds
\end{align}
for all $t \in [0,T]$ and a.s. in $\Omega$, where
\begin{align*}
&\langle -\operatorname{div}\,G+f,S'(u)\psi\rangle=\langle -\operatorname{div}\,G+f,S'(u)\psi\rangle_{W^{-1,p'}(D)+L^1(D), W^{1,p}_0(D)\cap L^{\infty}(D)}\\
&=\int_D(G\cdot\nabla[S'(u)\psi]+fS'(u)\psi)\, dx
\end{align*}
a.e. in $\Omega\times (0,T)$.
\end{prop}
\begin{rem}\label{Remark 4.2}
From Lemma 1.4 in \cite{Temam} we know that
\begin{align*}
L^{\infty}(0,T; L^1(D)) \cap \mathcal{C}_w([0,T]; L^1(D) + W^{-1,p'}(D)) \subset \mathcal{C}_w([0,T]; L^1(D)).
\end{align*}
Hence, in Proposition \ref{Proposition 4.2} we have $u \in \mathcal{C}_w([0,T]; L^1(D))$ a.s. in $\Omega$. Therefore, $u(t) \in L^1(D)$ makes sense for all $t \in [0,T]$, a.s. in $\Omega$.\\
Moreover, since the right-hand side of \eqref{2} is an element of $L^2(D)$ for all $t\in [0,T]$, even if the members on the left-hand are not in $L^2(D)$, \eqref{2} holds also in $L^2(D)$.
\end{rem}
\begin{proof}
See \cite{NSAZ}, Lemma 4.1. In the statement of this lemma $u$ requires to be an element of $L^1(\Omega; \mathcal{C}([0,T]; L^1(D)))$. But in the proof of this lemma it is only necessary that $u(t) \in L^1(D)$ makes sense for all $t \in [0,T]$, a.s. in $\Omega$ and this is the case because of Remark \ref{Remark 4.2}.
\end{proof}

\section{Renormalized solutions}\label{S5}
Let us assume that there exists a strong solution $u$ to \eqref{1} in the sense of Theorem \ref{Theorem 2.1}. We observe that for initial data $u_0$ merely in $L^1$, the It\^{o} formula for the square of the norm (see, e.g., \cite{EP}) can not be applied and consequently the natural a priori estimate for $\nabla u$ in $L^p(\Omega\times Q_T)^d$ is not available. Choosing $g=\Phi(u)$, $f\equiv 0$, $\psi\equiv 1$ and 
\[S(u)=\int_0^{u}T_k(r)\, dr\]
in \eqref{Itoformulat}, where $T_k:\mathbb{R}\rightarrow\mathbb{R}$ is the truncation function at level $k>0$ defined by
\begin{align*}
T_k(r)=\begin{cases} r &,~|r| \leq k, \\ 
k\operatorname{sign}(r) &,~ |r|> k,
\end{cases}
\end{align*}
we find that there exists a constant $C(k)\geq 0$ depending on the truncation level $k>0$, such that
\begin{align*}%\label{200826_02}
\mathbb{E}\int_0^T\int_D |\nabla T_k(u)|^p\,dx\,ds\leq C(k).
\end{align*}
As in the deterministic case, the notion of renormalized solutions takes this information into account : 
\begin{defn}\label{Definition 5.1}
Let $u_0 \in L^1(\Omega \times D)$ be $\mathcal{F}_0$-measurable. A progressively measurable process $u:\Omega\times (0,T)\rightarrow L^1(D)$ such that $u\in L^1(\Omega \times Q_T)$ is a renormalized solution to (\ref{1}) with initial value $u_0$, iff
\begin{itemize}
\item[(i)]$\operatorname{ess}\sup_{t\in (0,T)}\mathbb{E}\Vert u(t)\Vert_{L^1}<+\infty$ and
$T_k(u) \in L^p(\Omega; L^p(0,T;W_0^{1,p}(D)))$ for all $k>0$,
\item[(ii)]
For all $\psi \in \mathcal{C}^{\infty}([0,T] \times \bar{D})$ and all $S \in \mathcal{C}^2(\mathbb{R})$ such that $S'$ has compact support with $S'(0)=0$ or $\psi(t,x) =0$ for all $(t,x) \in [0,T] \times \partial D$ the equality
\begin{align}\label{reneq190204}
&\int_D S(u(t))\psi(t) - S(u_0) \psi(0) \,dx + \int_0^t \int_D S''(u) |\nabla u|^p \psi \,dx\,ds\nonumber \\
+ &\int_0^t \int_D S'(u) |\nabla u|^{p-2} \nabla u \cdot \nabla \psi \,dx\,ds\nonumber\\
= &\int_0^t \int_D S'(u) \psi \Phi(u) \,dx\,d\beta + \int_0^t \int_D S(u) \psi_t \,dx\,ds + \frac{1}{2} \int_0^t \int_D  S''(u) \psi \Phi(u)^2 \,dx\,ds
\end{align}
holds true a.s. in $\Omega\times (0,T)$.
\item[(iii)]
The following energy dissipation condition holds true:
\begin{align*}
\lim\limits_{k \to \infty}\mathbb{E}\int_{\{k < |u| < k+1 \}} |\nabla u|^p \,dx\,dt= 0.
\end{align*}
\end{itemize}
\end{defn}
\begin{rem}
For $u$ as in Definition \ref{Definition 5.1} such that $(i)$ is satisfied, $(ii)$ implies that for any $S\in\mathcal{C}^2(\mathbb{R})$ such that $S'$ has compact support with $S(0)=0$ there exists a version of $S(u)$ with paths in $\mathcal{C}([0,T];L^1(D)+W^{-1,p'}(D))$ and this version satisfies
\begin{align}\label{180625_01}
&S(u(t))-S(u(0))-\int_0^t \operatorname{div}\,(S'(u)|\nabla u|^{p-2}\nabla u) \,ds+\int_0^t S''(u)[|\nabla u|^p-\frac{1}{2}\Phi(u)^2] \,ds \nonumber\\
&=\int_0^t \Phi(u) S'(u) \,d\beta,
\end{align}
or equivalently, in differential form,
\begin{align}\label{SPDE1}
&dS(u)-\operatorname{div}\,(S'(u)|\nabla u|^{p-2}\nabla u) \,dt+S''(u)[|\nabla u|^p-\frac{1}{2}\Phi(u)^2]\,dt
=\Phi(u) S'(u) \,d\beta
\end{align}
in $W^{-1,p'}(D)+L^1(D)$ a.s. in $\Omega$ for any $t\in[0,T]$.  Since the right-hand side of \eqref{SPDE1} is in $L^2(D)$, the equation also holds in $L^2(D)$. From Definition \ref{Definition 5.1} it follows that $S(u)$ is bounded and therefore by Remark \ref{Remark 4.2} it follows that $S(u)\in \mathcal{C}_w([0,T];L^1(D))$.
\end{rem}
\begin{rem}\label{Remark 5.2}
If $u$ is a renormalized solution to \eqref{1}, thanks to \eqref{SPDE1}, the It\^{o} formula from Proposition \ref{Proposition 4.2} still holds true for $S(u)$ for any $S\in W^{2,\infty}(\mathbb{R})$ with $\operatorname{supp}(S')$ compact such that $S(u)\in W^{1,p}_0(D)$ a.s. in $\Omega\times (0,T)$. Indeed, in this case \eqref{2} is satisfied for the progressively measurable functions
\[\tilde{u}=S(u)\in L^p(\Omega; L^p(0,T;W_0^{1,p}(D)))\cap L^1(\Omega; \mathcal{C}([0,T]; L^1(D) + W^{-1,p'}(D))),\] 
\[G=S'(u)|\nabla u|^{p-2}\nabla u\in L^{p'}(\Omega\times Q_T)^d,\]
\[f=S''(u)[|\nabla u|^p-\frac{1}{2}\Phi(u)^2\in L^1(\Omega\times Q_T),\]
\[g=\Phi(u) S'(u)\in L^2(\Omega\times Q_T).\] 
\end{rem}
\begin{rem}
Let $u$ be a renormalized solution to \eqref{1} with $\nabla u\in L^p(\Omega\times Q_T)^d$. For fixed $l>0$, let $h_l:\mathbb{R}\rightarrow\mathbb{R}$ be defined by
\begin{align*}
h_l(r)=\begin{cases} 0 &,~|r| \geq l+1 \\ 
l+1-|r| &,~ l<|r|<l+1\\
1 &,~ |r|\leq l.
\end{cases}
\end{align*}
Taking  $S(u)=\int_0^{u} h_l(r) \,dr$ as a test function in \eqref{reneq190204}, we may pass to the limit with $l\rightarrow\infty$ and we find that $u$ is a strong solution to \eqref{1} (see also \cite{NSAZ}).
\end{rem}

\section{Existence of renormalized solutions}\label{S6}
In this Section, we fix $u_0 \in L^1(\Omega \times D)$ $\mathcal{F}_0$-measurable. Let $(u_0^n)_n \subset L^2(\Omega \times D)$ be an $\mathcal{F}_0$-measurable sequence such that $|u_0^n|\leq |u_0|$ for all $n\in\mathbb{N}$ and $\lim_{n\rightarrow\infty}u_0^n=u_0$ in $L^1(\Omega \times D)$ and in $L^1(D)$ for a.e. $\omega\in\Omega$. A possible choice is $u_0^n=T_n(u_0)$, $n\in\mathbb{N}$.
\begin{thm}\label{Theorem 6.2}
Let $\Phi$ be bounded. Then, there exists a renormalized solution to \eqref{1} with initial datum $u_0$.
\end{thm}
Theorem \ref{Theorem 6.2} will be proved succesively in the following Lemmas.
\begin{lem}\label{Lemma 2}
There exist constants $C(k), C(k,k')>0$ only depending on $k>0$ or $k,k'>0$, respectively, such that
\begin{itemize}
\item[(i)] $\mathbb{E} \int_0^T \int_D |\nabla T_k(u_n)| \,dx \,dt \leq C(k)$ for all $k>0$ and all $n \in \mathbb{N}$,
\item[(ii)] $\mathbb{E} \int_0^T \int_D |\nabla \theta_k^{k'}(u_n)| \,dx \,dt \leq C(k,k')$ for all $k,k'>0$ and all $n \in \mathbb{N}$, where $\theta_k^{k'}(r):= T_{k+k'}(r) - T_k(r)$ for all $r \in \mathbb{R}$.
\end{itemize}
\end{lem}
\begin{proof}
Since $u_n$ is a strong solution to \eqref{1} with initial value $u_0^n$, $u_n$ satisfies the equality
\begin{align}\label{un}
u_n(t) - u_0^n - \int_0^t \textnormal{div}\,(|\nabla u_n|^{p-2} \nabla u_n)\, ds = \int_0^t \Phi(u_n) \,d\beta
\end{align}
in $L^2(D)$ for all $t \in [0,T]$ and a.s. in $\Omega$. Applying Proposition \ref{Proposition 4.2} with $S= \int_0^{\cdot} T_k(r) \, dr$ and $\psi \equiv 1$ and taking the expectation yields
\begin{align}\label{200827_01}
&\mathbb{E} \int_D \int_0^{u_n(t)} T_k(r) \,drdx + \mathbb{E}\int_0^t \int_D |\nabla T_k(u_n)|^p \,dx\,ds\nonumber\\
= \frac{1}{2} &\mathbb{E} \int_0^t \int_D  T'_k(u_n) \Phi(u_n)^2 \,dx\,ds + \mathbb{E} \int_D \int_0^{u_0^n} T_k(r) \,dr\,dx
\end{align}
for all $k>0$, all $t \in [0,T]$ and a.s. in $\Omega$. The first term on the left hand side of \eqref{200827_01} is nonnegative. Now, the Lipschitz continuity of $\Phi$ and the estimate $|u_0^n| \leq |u_0|$ yield
\begin{align*}
\mathbb{E}\int_0^T \int_D |\nabla T_k(u_n)|^p \,dx\,ds \leq \frac{TL^2 k^2 |D|}{2} + k \Vert u_0 \Vert_{L^1(\Omega \times D)} =: C(k),
\end{align*}
where $L$ is the Lipschitz constant of $\Phi$. This proves $(i)$ and assertion $(ii)$ is a direct consequence of $(i)$.
\end{proof}
\begin{lem}\label{Lemma 3}
Passing to not relabelded subsequences if necessary, we have the following convergence results:
\begin{itemize}
\item[(i)] There exists a progressively measurable process $u: \Omega \times [0,T] \to L^1(D)$ such that $u \in L^1(\Omega \times Q_T)$, $u_n \to u $ in $L^1(\Omega \times Q_T)$ and $u_n(t) \to u(t)$ in $L^1(\Omega\times D)$ for a.e. $t\in (0,T)$ and in $L^1(D)$ a.e. in $\Omega\times (0,T)$.\\ 
Moreover, $\operatorname{ess}\sup_{t\in (0,T)}\mathbb{E}\Vert u(t)\Vert_{L^1}<+\infty$.
\item[(ii)] $\nabla T_k(u_n) \rightharpoonup \nabla T_k(u)$ in $L^p(\Omega \times Q_T)^d$ for all $k>0$,
\item[(iii)] $\nabla \theta_k^{k'}(u_n) \rightharpoonup \nabla \theta_k^{k'}(u)$ in $L^p(\Omega \times Q_T)^d$ for all $k,k'>0$,
\item[(iv)] There exists $\sigma_k \in L^{p'}(\Omega \times Q_T)^d$ such that $|\nabla T_k(u_n)|^{p-2} \nabla T_k(u_n) \rightharpoonup \sigma_k$ in $L^{p'}(\Omega \times Q_T)^d$ satisfying $\sigma_k= \sigma_{k+1} \chi_{\{|u|<k\}}$ on $\{|u| \neq k\}$ for all $k>0$.
\item[(v)] There exists $\tilde{\sigma}_k^{k'} \in L^{p'}(\Omega \times Q_T)^d$ such that $|\nabla \theta_k^{k'}(u_n)|^{p-2} \nabla \theta_k^{k'}(u_n) \rightharpoonup \tilde{\sigma}_k^{k'}$ in $L^{p'}(\Omega \times Q_T)^d$ satisfying $\tilde{\sigma}_k^{k'}= \tilde{\sigma}_{k-1}^{k'+2} \chi_{\{k<|u|<k+k'\}}$ on $\{|u| \neq k\} \cap \{|u| \neq k+k'\}$ for all $k,k'>0$.
\item[(vi)] We have
\begin{align}\label{23}
\lim\limits_{n,m \to \infty} \mathbb{E} \int_0^T \int_D (|\nabla u_n|^{p-2} \nabla u_n - |\nabla u_m|^{p-2} \nabla u_m) \cdot \nabla T_k(u_n - u_m) \,dx\,dt=0
\end{align}
for all $k>0$.
\end{itemize}
\end{lem}
\begin{rem}
Condition $(ii)$ of Lemma \ref{Lemma 3} implies that condition $(i)$ from Definition \ref{Definition 5.1} is satisfied for $u$.
\end{rem}
\begin{proof}
Theorem \ref{Theorem 3.1} yields that
\begin{align*}
\mathbb{E} \int_D |u_n(t) - u_m(t)| \,dx \leq \mathbb{E} \int_D |u_0^n - u_0^m| \,dx \to 0
\end{align*}
as $n,m \to \infty$ for all $t \in [0,T]$. Hence $u_n$ is a Cauchy sequence in $L^1(\Omega \times Q_T)$ and $u_n(t)$ is a Cauchy sequence in $L^1(\Omega \times D)$ for all $t \in [0,T]$. Hence there exists a progessively measurable process $u: \Omega \times [0,T] \to L^1(D)$ with $u\in L^1(\Omega \times Q_T)$ such that, for $n\rightarrow\infty$, $u_n\rightarrow u$ in $L^1(\Omega\times Q_T)$ and $u_n(t) \to u(t)$ in $L^1(\Omega \times D)$ a.e. in $(0,T)$, passing to a suitable subsequence if necessary, also in $L^1(D)$, a.s. in $\Omega\times (0,T)$.\\
Now we observe that, thanks to $(A1)$, $v\equiv 0$ is the unique solution of \eqref{1} with initial value $v_0\equiv 0$. Applying Theorem \ref{Theorem 3.1} with $u=u_n$, $u_0=u_0^n$ and $v=v_0\equiv 0$ it follows that
\[\sup_{t\in [0,T]}\mathbb{E}\Vert u_n(t)\Vert_{L^1}<+\infty\]
uniformly in $n\in\mathbb{N}$. Hence the last claim of assertion $(i)$ follows from Fatou's Lemma.\\
Since $\nabla T_k(u_n)$ is bounded in $L^p(\Omega \times Q_T)^d$ and $u_n \to u$ in $L^1(\Omega \times Q_T)$, assertion $(ii)$ follows. Using similar arguments we may conclude $(iii)$. In order to prove $(iv)$ and $(v)$ we only have to show that $\sigma_k= \sigma_{k+1} \chi_{\{|u|<k\}}$ on $\{|u| \neq k\}$ for all $k>0$ and $\tilde{\sigma}_k^{k'}= \tilde{\sigma}_{k-1}^{k'+2} \chi_{\{k<|u|<k+k'\}}$ on $\{|u| \neq k\} \cap \{|u| \neq k+k'\}$ for all $k,k'>0$. Let $\psi \in L^p(\Omega \times Q_T)^d$. Then
\begin{align*}
\lim\limits_{n \to \infty} \mathbb{E} \int_{Q_T} |\nabla T_k(u_n)|^{p-2} \nabla T_k(u_n) \cdot \psi \cdot \chi_{\{|u| \neq k \}} \,dx\,dt = \mathbb{E} \int_{Q_T} \sigma_k \psi \cdot \chi_{\{|u| \neq k \}} \,dx\,dt.
\end{align*}
On the other hand we know that $u_n \to u$ a.e. in $\Omega \times Q_T$ for a subsequence. Hence, we have $\chi_{\{|u_n|<k \}} \to \chi_{\{|u|<k \}}$ a.e. in $\{|u| \neq k \}$. Lebesgue's Theorem yields
\begin{align*}
\chi_{\{|u_n|<k \}} \cdot \chi_{\{|u| \neq k \}} \cdot \psi \to \chi_{\{|u|<k \}} \cdot \chi_{\{|u|\neq k \}} \cdot \psi ~~~\textnormal{in}~ L^p(\Omega \times Q_T)^d.
\end{align*}
We may conclude that
\begin{align*}
&\lim\limits_{n \to \infty} \mathbb{E} \int_{Q_T} |\nabla T_k(u_n)|^{p-2} \nabla T_k(u_n) \cdot \psi \cdot \chi_{\{|u| \neq k \}} \,dx\,dt  \\
= &\lim\limits_{n \to \infty} \mathbb{E} \int_{Q_T} |\nabla T_{k+1}(u_n)|^{p-2} \nabla T_{k+1}(u_n) \cdot \psi \cdot \chi_{\{|u| \neq k \}} \cdot \chi_{\{|u_n|<k \}} \,dx\,dt \\
= & \mathbb{E} \int_{Q_T} \sigma_{k+1} \psi \chi_{\{|u| \neq k \}} \cdot \chi_{\{|u|<k \}} \,dx\,dt.
\end{align*}
Therefore it follows that $\sigma_k= \chi_{\{|u| < k\}} \sigma_{k+1}$ a.e. on $\{|u| \neq k \}$ which concludes $(iv)$. Since we have $\nabla \theta_k^{k'}(u_n) = \nabla T_{k+k'}(u_n) \chi_{\{k < |u_n| < k+k'\}}$, similar reasoning as in $(iv)$ yields $(v)$.\\
In order to prove $(vi)$ we apply Proposition \ref{Proposition 4.2} with $S= \tilde{T}_k:= \int_0^{\cdot} T_k(r) \, dr$ and $\psi \equiv 1$ to the equality
\begin{align}\label{blaa}
&u_n(t) - u_m(t) = u_0^n - u_0^m \\
+ &\int_0^t \textnormal{div} (|\nabla u_n|^{p-2} \nabla u_n - |\nabla u_m|^{p-2} \nabla u_m) \,ds + \int_0^t \Phi(u_n) - \Phi(u_m) \, d\beta. \notag
\end{align}
Applying the expectation afterwards yields
\begin{align*}
&\mathbb{E} \int_D \tilde{T}_k(u_n(T) - u_m(T)) \,dx \\
+ &\mathbb{E} \int_0^T \int_D (|\nabla u_n|^{p-2} \nabla u_n - |\nabla u_m|^{p-2} \nabla u_m) \cdot \nabla T_k(u_n - u_m) \,dx\,dt \\
= &\mathbb{E} \int_D \tilde{T}_k(u_0^n - u_0^m) \,dx 
+ \frac{1}{2} \mathbb{E}\int_0^T \int_D \chi_{\{|u_n-u_m|<k\}} (\Phi(u_n) - \Phi(u_m))^2 \, dx \, dt.
\end{align*}
Since the first term on the left hand side is nonnegative and $\chi_{\{|u_n-u_m|<k\}} (\Phi(u_n) - \Phi(u_m))^2 \leq L^2k^2$ a.e. on $\Omega \times Q_T$, where $L$ is a Lipschitz constant of $\Phi$, Lebesgue's Dominated Convergence Theorem yields assertion $(vi)$.
\end{proof}
In the following, we will show that the process $u$ from Lemma \ref{Lemma 3} $(i)$ is a renormalized solution to \eqref{1} with initial datum $u_0$ in the sense of Definition \ref{Definition 5.1}.
\begin{lem}\label{200827_lem1}
The function $u$ from Lemma \ref{Lemma 3} satisfies condition $(ii)$ from Definition \ref{Definition 5.1}.
\end{lem}

\begin{proof}
Let $u_n$ be a strong solution to (\ref{1}) with initial value $u_0^n$, i.e.,
\begin{align}\label{20}
u_n(t) - u_0^n - \int_0^t \textnormal{div}\, (|\nabla u_n|^{p-2} \nabla u_n) \,ds = \int_0^t \Phi(u_n) \,d\beta
\end{align}
for all $t \in [0,T]$ and a.s. in $\Omega$. 
We apply the It\^{o} formula introduced in Proposition \ref{Proposition 4.2} to equality \eqref{20}. Therefore we know that for all $\psi \in C^{\infty}([0,T] \times \overline{D})$ and all $S \in W^{2,\infty}(\mathbb{R})$ such that $S''$ is piecewise continuous and $S'(0)=0$ or $\psi(t,x) =0$ for all $(t,x) \in [0,T] \times\partial D$ the equality
\begin{align}\label{190206_03}
&\int_D S(u_n(t))\psi(t) - S(u_0^n) \psi(0) \,dx + \int_0^t \int_D S''(u_n) |\nabla u_n|^p \psi \,dx\,ds \notag \\
&+ \int_0^t \int_D S'(u_n) |\nabla u_n|^{p-2} \nabla u_n \cdot \nabla \psi \,dx\,ds\\
= &\int_0^t \int_D S'(u_n) \psi \Phi(u_n) \,dx\,d\beta + \int_0^t \int_D S(u_n) \psi_t \,dx\,ds \notag \\
+ &\frac{1}{2} \int_0^t \int_D  S''(u_n) \psi \Phi(u_n)^2 \,dx\,ds \notag
\end{align}
holds true for all $t \in [0,T]$ and a.s. in $\Omega$. In the following, passing to a suitable, not relabeled subsequence if necessary, and taking the limit for $n\rightarrow\infty$, we will show that \eqref{190206_03} is also satisfied by $u$ and $u_0$ respectively and therefore $(ii)$ from Definition \ref{Definition 5.1} holds. To this end, it is left to show that
\begin{align*}
T_k(u_n) \to T_k(u)~~~ \textnormal{in} ~~~L^p(\Omega; L^p(0,T;W_0^{1,p}(D)))
\end{align*}
for all $k>0$.
Similar as in \cite{NSAZ}, the upcoming technical lemmas which are inspired by Theorem 2 and Lemma 2 in \cite{DB} will show that this convergence holds.
\end{proof}
\begin{lem}\label{Lemma 5.6}
For $n\in\mathbb{N}$, let $u_n$ be a strong solution to \eqref{1} with respect to the initial value $u_0^n$. Let $H$ and $Z$ be two real valued functions belonging to $W^{2,\infty}(\mathbb{R})$ such that $H''$ and $Z''$ are piecewise continuous, $H'$ and $Z'$ have compact supports and $Z(0)=Z'(0)=0$ is satisfied. Then
\begin{align}\label{7}
\lim\limits_{n,m \to \infty} \mathbb{E} \int_0^T \int_D H''(u_n) Z(u_n - u_m) |\nabla u_n|^p \,dx\,dt =0.
\end{align}
\end{lem}
\begin{proof}
Using the It\^{o} product rule (see Proposition \ref{itoproduct}) yields
\begin{align}\label{14}
&\int_D Z(u_n(t) - u_m(t)) H(u_n(t)) \,dx= \int_D Z(u_0^n - u_0^m)H(u_0^n) \,dx \notag\\
- &\int_0^t \int_D |\nabla u_n|^{p-2} \nabla u_n \nabla \bigg(Z(u_n - u_m) H'(u_n) \bigg) \,dx\,ds \notag\\
+ &\frac{1}{2} \int_0^t \int_D H''(u_n)Z(u_n - u_m) \Phi(u_n)^2 \,dx\,ds \notag\\
+ &\int_0^t \int_D H'(u_n)Z(u_n - u_m) \Phi(u_n) \,dx\,d\beta \notag\\
- &\int_0^t \int_D (|\nabla u_n|^{p-2} \nabla u_n - |\nabla u_m|^{p-2} \nabla u_m) \nabla \bigg( Z'(u_n - u_m)H(u_n) \bigg) \,dx\,ds\\
+ &\frac{1}{2} \int_0^t \int_D H(u_n) Z''(u_n - u_m)(\Phi(u_n) - \Phi(u_m))^2 \, dx \, ds \notag\\
+ &\int_0^t \int_D H(u_n) Z'(u_n - u_m) (\Phi(u_n) - \Phi(u_m)) \, dx \, d\beta \notag\\
+ &\int_0^t \int_D Z'(u_n - u_m) H'(u_n) (\Phi(u_n) - \Phi(u_m)) \Phi(u_n) \, dx \, ds \notag
\end{align}
for all $t \in [0,T]$ and a.s. in $ \Omega$. Now we set $t=T$ and take the expectation in equality \eqref{14}. Since $Z',Z'',H'$ and $H''$ have compact support and $\Phi$ is Lipschitz continuous it is easy to see that
\begin{align*}
\lim\limits_{n,m \to \infty} \mathbb{E} \int_0^T \int_D H(u_n) Z''(u_n - u_m)(\Phi(u_n) - \Phi(u_m))^2 \, dx \, dt =0,
\end{align*}
\begin{align*}
\lim\limits_{n,m \to \infty} \mathbb{E} \int_0^T \int_D Z'(u_n - u_m) H'(u_n) (\Phi(u_n) - \Phi(u_m)) \Phi(u_n) \, dx \, dt =0
\end{align*}
and
\begin{align*}
\lim\limits_{n,m \to \infty} \mathbb{E} \int_0^T \int_D H''(u_n)Z(u_n - u_m) \Phi(u_n)^2 \,dx\,ds =0.
\end{align*}
From now on the proof is the same as in \cite{NSAZ} or in \cite{DB}, Theorem 2.
\end{proof}
\begin{lem}\label{Lemma 6.3}
For $n\in\mathbb{N}$, let $u_n$ be a strong solution to \eqref{1} with respect to the initial value $u_0^n$. Let $u$ be defined as in Lemma \ref{Lemma 3}. Then,
\begin{align}\label{22}
\lim\limits_{n,m \to \infty} \mathbb{E} &\int_0^T \int_D \bigg( |\nabla T_k(u_n)|^{p-2} \nabla T_k(u_n) - |\nabla T_k(u_m)|^{p-2} \nabla T_k(u_m) \bigg) \cdot \nonumber \\
&\cdot( \nabla T_k(u_n) - \nabla T_k(u_m)) \,dx\,ds=0.
\end{align}
Especially, we have
\begin{align*}
\nabla T_k(u_n) \to \nabla T_k(u)~~~ \textnormal{in} ~~~L^p(\Omega \times Q_T)^d
\end{align*}
and
\begin{align*}
T_k(u_n) \to T_k(u)~~~ \textnormal{in} ~~~L^p(\Omega; L^p(0,T;W_0^{1,p}(D)))
\end{align*}
for $n \to \infty$ and for all $k>0$.
\end{lem}

\begin{proof}
For $k>0$, we set 
\begin{align*}
&\int_0^T \int_D \bigg( |\nabla T_k(u_n)|^{p-2} \nabla T_k(u_n) - |\nabla T_k(u_m)|^{p-2} \nabla T_k(u_m) \bigg) \cdot  \\
&\cdot( \nabla T_k(u_n) - \nabla T_k(u_m)) \,dx\,dt \\
&= I_k^{n,m} + J_k^{n,m} + J_k^{m,n},
\end{align*}
a.s. in $\Omega$, where
\begin{align*}
I_k^{n,m} &= \int_{\{|u_n| \leq k\} \cap \{|u_m| \leq k\}} (|\nabla u_n|^{p-2} \nabla u_n - |\nabla u_m|^{p-2} \nabla u_m) \cdot \nabla (u_n - u_m) \,dx\,dt, \\
J_k^{n,m} &= \int_{\{|u_n| \leq k\} \cap \{|u_m| > k\}} |\nabla u_n|^{p-2} \nabla u_n \cdot \nabla u_n \,dx\,dt
\end{align*}
a.s. in $\Omega$. $J_k^{m,n}$ is the same as $J_k^{n,m}$ where the roles of $n$ and $m$ are reversed. Therefore these two terms can be treated simultaneously.\\
Moreover, we set 
\begin{align*}
0 \leq J_k^{n,m} = J_{1,k,k'}^{n,m} + J_{2,k,k'}^{n,m},
\end{align*}
where
\begin{align*}
J_{1,k,k'}^{n,m} &= \int_{\{|u_n| \leq k\} \cap \{|u_m| > k\} \cap \{|u_n - u_m| \leq k'\}} |\nabla u_n|^{p-2} \nabla u_n \cdot \nabla u_n \,dx\,dt,\\
J_{2,k,k'}^{n,m} &=\int_{\{|u_n| \leq k\} \cap \{|u_m| > k\} \cap \{|u_n - u_m| > k'\}} |\nabla u_n|^{p-2} \nabla u_n \cdot \nabla u_n \,dx\,dt
\end{align*}
for all $k'>k>0$, a.s. in $\Omega$. Since Lemma \ref{Lemma 3} $(iii), (iv), (v)$ and $(vi)$ hold true, the same arguments as in \cite{NSAZ} yield 
\begin{align*}
0= \lim\limits_{n,m \to \infty} \mathbb{E} I_k^{n,m} = \lim\limits_{n,m \to \infty} \mathbb{E} J_{1,k,k'}^{n,m}.
\end{align*}
The proof that
\begin{align*}
\lim\limits_{n,m \to \infty} \mathbb{E} J_{2,k,k'}^{n,m}=0
\end{align*}
is slightly different from \cite{NSAZ} and therefore we will give the proof of this equality.
We use Lemma \ref{Lemma 5.6}, \eqref{7} with $H=H^{\delta}_k$ for $\delta,k>0$ such that
\begin{align*}
(H_k^{\delta})''(r)=
\begin{cases}
1, ~&|r|<k, \\
-k \delta, ~&k \leq |r| \leq k+ \frac{1}{\delta}, \\
0, ~&|r| > k + \frac{1}{\delta}
\end{cases}
\end{align*}
and get
\begin{align*}
&\limsup\limits_{n\to \infty} \limsup\limits_{m \to \infty}  \mathbb{E}\int_{\{|u_n| \leq k\}} Z(u_n - u_m) |\nabla u_n|^p \,dx\,dt\,\\
\leq &\delta \cdot k \limsup\limits_{n\to \infty} \limsup\limits_{m \to \infty}  \mathbb{E} \int_{\{k \leq |u_n| \leq k+\frac{1}{\delta}\}} Z(u_n - u_m) |\nabla u_n|^p \,dx\,dt\, \\
\leq &\delta \cdot k \Vert Z \Vert_{\infty} \limsup\limits_{n \to \infty}  \mathbb{E} \int_{\{k \leq |u_n| \leq k+\frac{1}{\delta}\}} |\nabla u_n|^p \,dx\,dt\,.
\end{align*}
Now applying Proposition \ref{Proposition 4.2} to the equation of $u_n$ with $S= \int_0^{\cdot} \theta_{k}^{\frac{1}{\delta}} =: \tilde{\theta}_k^{\frac{1}{\delta}}$, $\psi \equiv 1$, $g=\Phi(u_n)$ and $f \equiv0$ and taking the expectation yields
\begin{align*}
&\mathbb{E} \int_D \tilde{\theta}_k^{\frac{1}{\delta}}(u_n(T)) \,dx + \mathbb{E} \int_0^T \int_D \chi_{\{k \leq |u_n| \leq k + \frac{1}{\delta}\}} |\nabla u_n|^p \,dx \,dt \\
= &\mathbb{E} \int_D \tilde{\theta}_k^{\frac{1}{\delta}}(u_0^n) \,dx + \frac{1}{2} \mathbb{E} \int_0^T \int_D \chi_{\{k \leq |u_n| \leq k + \frac{1}{\delta}\}} \Phi(u_n)^2 \,dx \, dt.
\end{align*}
The first term on the left hand side is nonnegative and the integrand of the second term on the right hand side is bounded since $\Phi$ is bounded.\\
Multiplying by $\delta$ and passing to the limit with $n \to \infty$ yields
\begin{align*}
\delta \cdot \limsup\limits_{n \to \infty}~ \mathbb{E} \int_0^T \int_D \chi_{\{k \leq |u_n| \leq k + \frac{1}{\delta}\}} |\nabla u_n|^p \,dx \,dt \leq \mathbb{E} \int_D \delta \tilde{\theta}_k^{\frac{1}{\delta}}(u_0) \,dx + \frac{1}{2} \delta \cdot C
\end{align*}
for a constant $C>0$.
We can estimate that $\delta \tilde{\theta}_k^{\frac{1}{\delta}}(u_0) \to 0$ a.e. in $\Omega \times D$ as $\delta \to 0$ and $|\delta \tilde{\theta}_k^{\frac{1}{\delta}}(u_0)| \leq u_0 + \tilde{C}$ for a constant $\tilde{C}>0$. Therefore Lebesgue's Theorem yields
\begin{align*}
\lim\limits_{\delta \to 0} \limsup\limits_{n \to \infty} \limsup\limits_{m \to \infty} \delta \cdot \mathbb{E} \int_{\{k \leq |u_n| \leq k+\frac{1}{\delta}\}} Z(u_n - u_m) |\nabla u_n|^p \,dx\,dt\, =0. 
\end{align*}
Thus we may conclude
\begin{align*}
&\lim\limits_{n,m \to \infty}  \mathbb{E}\int_{\{|u_n| \leq k\}} Z(u_n - u_m) |\nabla u_n|^p \,dx\,dt\,\\
= &\limsup\limits_{n\to \infty} \limsup\limits_{m \to \infty}  \mathbb{E}\int_{\{|u_n| \leq k\}} Z(u_n - u_m) |\nabla u_n|^p \,dx\,dt\,=0.
\end{align*}
Choosing $Z$ such that $Z(r)=1$ for $|r|\geq k'$ and $Z \geq 0$ on $\mathbb{R}$ such that $Z(0)=Z'(0)=0$, it follows
\begin{align*}
0 &\leq \lim\limits_{n,m \to \infty} \mathbb{E} J_{2,k,k'}^{n,m}\\
&= \lim\limits_{n,m \to \infty} \mathbb{E} \int_{\{|u_n| \leq k\} \cap \{|u_m| > k\} \cap \{|u_n - u_m| > k'\}} |\nabla u_n|^{p-2} \nabla u_n \cdot \nabla u_n \,dx\,dt \\
&\leq \lim\limits_{n,m \to \infty}  \mathbb{E} \int_{\{|u_n| \leq k\}} Z(u_n - u_m) |\nabla u_n|^p \,dx\,dt=0,
\end{align*}
which finally shows the validity of equality \eqref{22}.
Since equality \eqref{22} holds true, it follows that
\begin{align}\label{25}
\lim\limits_{n \to \infty} \mathbb{E} \int_0^T \int_D |\nabla T_k(u_n)|^{p-2} \nabla T_k(u_n) \cdot \nabla T_k(u_n) \,dx\,dt= \mathbb{E} \int_0^T \int_D \sigma_k \cdot \nabla T_k(u) \,dx\,dt.
\end{align}
Minty's trick yields $\sigma_k= |\nabla T_k(u)|^{p-2} \nabla T_k(u)$. We may conclude by using equality \eqref{25} that
\begin{align*}
\lim\limits_{n \to \infty} \Vert \nabla T_k(u_n) \Vert_{L^p(\Omega \times Q_T)^d}^p = \Vert \nabla T_k(u)\Vert_{L^p(\Omega \times Q_T)^d}^p.
\end{align*}
Since $L^p(\Omega \times Q_T)^d$ is uniformly convex and $\nabla T_k(u_n) \rightharpoonup \nabla T_k(u)$ in $L^p(\Omega \times Q_T)^d$ it yields
\begin{align*}
\nabla T_k(u_n) \to \nabla T_k(u) ~~~\textnormal{in} ~ L^p(\Omega \times Q_T)^d
\end{align*} 
which ends the proof of Lemma \ref{Lemma 6.3}. 
\end{proof}
For the proof of Theorem \ref{Theorem 6.2} is left to show that the energy dissipation condition $(iii)$ from Definition \ref{Definition 5.1} holds true. To this end we have to show the following lemma at first.
\begin{lem}\label{190206_lem01}
For $n\in\mathbb{N}$, let $u_n$ be a strong solution to \eqref{1} with respect to the initial value $u_0^n$. Let $u$ be defined as in Lemma \ref{Lemma 3}. Then,
\begin{align}\label{190206_01}
\lim_{k\rightarrow\infty}\limsup_{n\rightarrow\infty}\,\mathbb{E}\int_{\{k<|u_n|<k+1\}} |\nabla u_n|^p \,dx\,dt=0.
\end{align}
\end{lem}
\begin{proof}
For fixed $l>0$, let $h_l:\mathbb{R}\rightarrow\mathbb{R}$ be defined as in Remark \ref{Remark 5.2}. We plug $S(r)=\int_0^{r} h_l(\overline{r})(T_{k+1}(\overline{r})-T_k(\overline{r}))\,d\overline{r}$ and $\psi \equiv 1$ in \eqref{190206_03} and take the expectation to obtain
\begin{align}\label{190206_04}
I_1+I_2+I_3=I_4+I_5,
\end{align}
where
\begin{align*}
I_1&=\mathbb{E}\int_D \int_{u_0^n}^{u_n(t)}h_l(r)(T_{k+1}(r)-T_k(r))\, dr\,dx,\\
I_2&=\mathbb{E}\int_{\{l<|u_n|<l+1\}}-\operatorname{sign}(u_n)(T_{k+1}(u_n)-T_k(u_n))|\nabla u_n|^p\,dx\,ds,\\
I_3&=\mathbb{E}\int_{\{k<|u_n|<k+1\}}h_l(u_n)|\nabla u_n|^p\,dx\,ds,\\
I_4&=\frac{1}{2}\mathbb{E}\int_{\{l<|u_n|<l+1\}} -\operatorname{sign}(u_n)(T_{k+1}(u_n)-T_k(u_n))\Phi(u_n)^2\,dx\,ds,\\
I_5&=\frac{1}{2}\mathbb{E}\int_{\{k<|u_n|<k+1\}}h_l(u_n)\Phi(u_n)^2\,dx\,ds
\end{align*}
for all $t\in [0,T]$.
By Lebesgue's Dominated Convergence theorem we may pass to the limit with $l\rightarrow\infty$ in \eqref{190206_04} and obtain
\begin{align}\label{190206_05}
J_1+J_2=J_3
\end{align}
where
\begin{align*}
J_1&=\mathbb{E}\int_D\int_{u_0^n}^{u_n(t)}T_{k+1}(r)-T_k(r)\, dr\,dx,\\
J_2&=\mathbb{E}\int_{\{k<|u_n|<k+1\}}|\nabla u_n|^p\,dx\,ds,\\
J_3&=\frac{1}{2} \mathbb{E}\int_{\{k<|u_n|<k+1\}}\Phi(u_n)^2\,dx\,ds.
\end{align*}
Passing to a not relabeled subsequence which may depend on $t$, Lemma \ref{Lemma 3} $(i)$ yields that $u_n(t) \rightarrow u(t)$ in $L^1(\Omega \times D)$ and $u_0^n \rightarrow u_0$ in $L^1(\Omega\times D)$ for $n\rightarrow \infty$. It follows that
\begin{align}\label{190206_08}
\lim_{k\rightarrow\infty}\lim_{n\rightarrow\infty}J_1=\lim_{k\rightarrow\infty} \int_D \int_{u_0}^{u(t)}T_{k+1}(r)-T_k(r)\, dr\,dx=0.
\end{align}
Moreover, we have
\begin{align*}
&\limsup\limits_{n \to \infty} \mathbb{E}\int_{\{k<|u_n|<k+1\}}\Phi(u_n)^2\,dx\,ds \leq \limsup\limits_{n \to \infty} \mathbb{E}\int_{\{k<|u_n|<k+1\}} C^2\,dx\,ds \\
\leq &\mathbb{E}\int_{\{k-1<|u|<k+2\}}C^2\,dx\,ds \to 0
\end{align*}
as $k \to \infty$, where $C>0$ is a constant depending on $\Phi$. Hence we get
\begin{align}\label{221}
\lim\limits_{k \to \infty} \limsup\limits_{n \to \infty} J_3 =0.
\end{align}
The nonnegativity of $J_2$, \eqref{190206_05}, \eqref{190206_08} and \eqref{221} yield
\begin{align*}
\lim\limits_{k \to \infty} \limsup\limits_{n \to \infty} J_2 =0
\end{align*}
and we may conclude \eqref{190206_01}.
\end{proof}
Now we can finalize the proof of Theorem \ref{Theorem 6.2}. We have 
\[\chi_{\{k<|u_n|<k+1\}}\chi_{\{|u|\neq k\}}\chi_{\{|u|\neq k+1\}}\rightarrow \chi_{\{k<|u|<k+1\}}\chi_{\{|u|\neq k\}}\chi_{\{|u|\neq k+1\}}\]
for $n\rightarrow\infty$ in $L^r(\Omega\times Q_T)$ for any $1\leq r<\infty$ and a.e. in $\Omega\times Q_T$. From Lemma \ref{Lemma 6.3} we recall that for any $k>0$,
\begin{align*}
\nabla T_k(u_n) \to \nabla T_k(u) ~~~\textnormal{in} ~ L^p(\Omega \times Q_T)^d
\end{align*}
for $n\rightarrow\infty$, thus, passing to a not relabeled subsequence if necessary, also a.s. in $\Omega\times Q_T$. Since $\nabla T_k(u)=0$ a.s. on $\{|u|=m\}$ for any $m\geq 0$, Fatou's Lemma yields 
\begin{align}\label{190206_02}
&\liminf_{n\rightarrow\infty}\, \mathbb{E}\int_{\{k<|u_n|<k+1\}} |\nabla u_n|^p \,dx\,dt\nonumber\\
&\geq \liminf_{n\rightarrow\infty} \mathbb{E}\,\int_{\{k<|u_n|<k+1\}} |\nabla u_n|^p\chi_{\{|u|\neq k\}}\chi_{\{|u|\neq k+1\}} \,dx\,dt\nonumber\\
&\geq\mathbb{E}\int_{\{k<|u|<k+1\}} |\nabla u|^p \chi_{\{|u|\neq k\}}\chi_{\{|u|\neq k+1\}} \,dx\,dt\nonumber\\
&=\mathbb{E}\int_{\{k<|u|<k+1\}} |\nabla u|^p \,dx\,dt
\end{align}
and the energy dissipation condition $(iii)$ follows combining Lemma \ref{190206_lem01} with \eqref{190206_02}.

\section{Uniqueness of renormalized solutions}\label{S7}
In the following, we formulate a contraction principle that yields immediately both uniqueness and continuous dependence on the initial values for renormalized solutions.
\begin{thm}\label{Theorem 7.1}
Let $u,v$ be renormalized solutions to \eqref{1} with initial data $u_0 \in L^1(\Omega \times D)$ and $v_0 \in L^1(\Omega \times D)$, respectively. Then,
\begin{align}\label{34}
\int_D \mathbb{E}|u(t) - v(t)| \,dx \leq \int_D \mathbb{E}|u_0 - v_0| \,dx
\end{align}
for all $t \in [0,T]$.
\end{thm}
\begin{proof}
This proof is inspired by the uniqueness proof in \cite{DBFMHR}. We know that $S(u)$ satisfies the SPDE
\begin{align}\label{35}
&dS(u)-\operatorname{div}\,(S'(u)|\nabla u|^{p-2}\nabla u) \,dt+S''(u)|\nabla u|^p\,dt\nonumber\\
&=\Phi S'(u) \,d\beta +\frac{1}{2}S''(u)\Phi^2(u) \,dt
\end{align}
for all $S \in C^2(\mathbb{R})$ such that $\textnormal{supp}~ S'$ compact and $S(0)=0$. Moreover, $S(v)$ satisfies an analogous SPDE. Subtracting both equalities yields
\begin{align}\label{36}
&S(u(t)) - S(v(t)) =\notag\\ 
&S(u_0) - S(v_0) + \int_0^t \textnormal{div}[S'(u) |\nabla u|^{p-2} \nabla u - S'(v)|\nabla v|^{p-2} \nabla v]\,ds  \notag \\
- &\int_0^t \left(S''(u)|\nabla u|^p  - S''(v)|\nabla v|^p\right)\,ds + \int_0^t (\Phi(u)S'(u) - \Phi(v)S'(v))  \,d\beta \\
+ &\frac{1}{2} \int_0^t (\Phi^2(u)S''(u) - \Phi^2(v)S''(v)) \,ds \notag
\end{align}
in $W^{-1,p'}(D) + L^1(D)$ for all $t \in [0,T]$, a.s. in $\Omega$.\\
Now we set $S(r):= T_s^{\sigma}(r)$ for $r \in \mathbb{R}$ and $s,\sigma >0$ and define $T_s^{\sigma}$ as follows: Firstly, we define for all $r \in \mathbb{R}$
\begin{align*}
(T_s^{\sigma})'(r)= \begin{cases}
1 ,~&\textnormal{if}~|r| \leq s, \\
\frac{1}{\sigma}(s+\sigma - |r|), ~&\textnormal{if}~ s < |r| < s + \sigma, \\
0, ~&\textnormal{if}~ |r| \geq s+ \sigma.
\end{cases}
\end{align*}
Then we set $T_s^{\sigma}(r):= \int_0^r (T_s^{\sigma})'(\tau) \,d\tau$. Furthermore we have the weak derivative
\begin{align*}
(T_s^{\sigma})''(r)= \begin{cases}
-\frac{1}{\sigma} \operatorname{sign}(r), ~&\textnormal{if}~ s < |r| < s + \sigma, \\
0, ~&\textnormal{otherwise}.
\end{cases}
\end{align*}
Applying the It\^{o} formula (see \ref{Proposition 4.2}) to equality \eqref{36} with $S(r)= \frac{1}{k} \tilde{T}_k(r)=\frac{1}{k}\int_0^r T_k(\overline{r})\,d\overline{r}$ and $\psi\equiv1$ yields 
\begin{align}\label{37}
&\int_D \bigg( \frac{1}{k} \tilde{T}_k( T_s^{\sigma}(u(t)) - T_s^{\sigma}(v(t))) - \frac{1}{k} \tilde{T}_k( T_s^{\sigma}(u_0) - T_s^{\sigma}(v_0)) \bigg) \,dx \notag \\
- &\int_0^t \langle \textnormal{div} ((T_s^{\sigma})'(u) |\nabla u|^{p-2} \nabla u - (T_s^{\sigma})'(v) |\nabla v|^{p-2} \nabla v), \frac{1}{k} T_k( T_s^{\sigma}(u) - T_s^{\sigma}(v)) \rangle \,dr \notag \\
= &\int_D \int_0^t ( - ((T_s^{\sigma})''(u) |\nabla u|^p - (T_s^{\sigma})''(v) |\nabla v|^p)\cdot \frac{1}{k} T_k (T_s^{\sigma}(u) - T_s^{\sigma}(v))\,dr \,dx\notag \\
+ &\int_D \int_0^t (\Phi(u)(T_s^{\sigma})'(u) - \Phi(v)(T_s^{\sigma})'(v)) \cdot\frac{1}{k} T_k (T_s^{\sigma}(u) - T_s^{\sigma}(v)) \,d\beta\,dx \notag\\
+& \frac{1}{2}\int_D \int_0^t (\Phi^2(u)(T_s^{\sigma})''(u) -\Phi^2(v)(T_s^{\sigma})''(v)) 
\cdot \frac{1}{k} T_k (T_s^{\sigma}(u) - T_s^{\sigma}(v)) \,dr \,dx \notag \\
+ &\frac{1}{2} \int_D \int_0^t (\Phi(u)(T_s^{\sigma})'(u) - \Phi(v)(T_s^{\sigma})'(v))^2 \cdot\frac{1}{k}  \chi_{\{|T_s^{\sigma}(u) - T_s^{\sigma}(v)|<k\}} \,dr\,dx
\end{align}
a.s. in $\Omega$ for any $t\in [0,T]$. We write equality \eqref{37} as
\begin{align*}
I_1^{\sigma,k,s} + I_2^{\sigma,k,s}= I_3^{\sigma,k,s}+ I_4^{\sigma,k,s} + I_5^{\sigma,k,s}+I_6^{\sigma,k,s}.
\end{align*}
We want to pass to the limit with $\sigma \to 0$ firstly, then we pass to the limit $k \to 0$ and finally we let $s \to \infty$.
We may repeat the arguments used in \cite{NSAZ}, proof of Theorem 7.1 for the expressions  $I_1$, $I_2$ and $I_3$ to pass to the limit in $I_1^{\sigma,k,s}$, $I_2^{\sigma,k,s}$ and $I_3^{\sigma,k,s}$ For $\omega\in \Omega$ and $t\in [0,T]$ fixed. More precisely, from \cite{NSAZ}, p.21 it follows that
\begin{align}\label{201211_01}
\lim\limits_{s \to \infty}\lim\limits_{k \to 0} \lim\limits_{\sigma \to 0} I_1 = \int_D |u(t) - v(t)| - |u_0 - v_0| \,dx
\end{align}
a.s. in $\Omega$, for all $t \in [0,T]$.
Then, repeating the arguments from \cite{NSAZ}, p.21 we get  
\begin{align}\label{201211_02}
\liminf\limits_{\sigma \to 0} I_2^{\sigma,k,s} \geq 0.
\end{align}
With the same arguments as on p.22-p.25 in \cite{NSAZ} it follows that
\begin{align}\label{201211_03}
\lim\limits_{j \to \infty} \limsup\limits_{k \to 0} \limsup\limits_{\sigma \to 0} |I_3^{\sigma,k,s_j}|=0,
\end{align}
passing to a suitable subsequence $(s_j)_{j\in\mathbb{N}}$ with $\lim_{j\rightarrow\infty}s_j=+\infty$ if necessary.\\
In the next steps, we will address $I_4^{\sigma,k,s}$, $I_5^{\sigma,k,s}$ and $I_6^{\sigma,k,s}$. Therefore we recall that for any fixed $s>0$, $T_s^{\sigma}(r)\rightarrow T_s(r)$ and $(T_s^{\sigma})'(r) \to \chi_{\{|r|\leq s\}}$ pointwise for all $r \in \mathbb{R}$ as $\sigma \to 0$. Since  $|(T_s^{\sigma})'| \leq 1$ and $|T_s^{\sigma})(r)|\leq |r|$ on $\mathbb{R}$ we have
\begin{align*}
(T_s^{\sigma})'(u) \to \chi_{\{|u|\leq s\}}
\end{align*}
in $L^1(\Omega\times Q_T)$ and a.e. in $\Omega \times Q_T$ 
as $\sigma \to 0$. An analogous result holds true for $v$ instead of $u$. Moreover, $\frac{1}{k}T_k(r)\rightarrow \operatorname{sign}(r)$ for $k\rightarrow 0$ in $\mathbb{R}$, where $\operatorname{sign}$ is the classical (single-valued) sign function. In addition, $|\frac{1}{k}T_k(r)|\leq 1$ for all $k>0$ and all $r\in \mathbb{R}$.\\
Now, we write 
\begin{align*}
I_4^{\sigma,k,s}=I_{4,1}^{\sigma,k,s}+I_{4,2}^{\sigma,k,s}
\end{align*}
where
\begin{align*}
I_{4,1}^{\sigma,k,s}&=\int_D \int_0^t \Phi(u)((T_s^{\sigma})'(u) - (T_s^{\sigma})'(v)) \cdot\frac{1}{k} T_k (T_s^{\sigma}(u) - T_s^{\sigma}(v)) \,d\beta\,dx,\\
I_{4,2}^{\sigma,k,s}&=\int_D \int_0^t (T_s^{\sigma})'(v)(\Phi(u) - \Phi(v))\cdot\frac{1}{k} T_k (T_s^{\sigma}(u) - T_s^{\sigma}(v)) \,d\beta\,dx
\end{align*}
Again, proceeding as in \cite{NSAZ}, p.25 for the term $I_4$ and using the boundedness of $\Phi$, it follows that
\begin{align*}
\lim_{s\rightarrow\infty}\lim_{k\rightarrow 0}\lim_{\sigma\rightarrow 0}I_{4,1}^{\sigma,k,s}=0
\end{align*}
a.s. in $\Omega$ for all $t\in[0,T]$. From the It\^{o} isometry and Lebesgue's dominated convergence theorem it follows that
\begin{align*}
&\lim_{\sigma\rightarrow 0}\mathbb{E}\left|I_{4,2}^{\sigma,k,s}-\int_D \int_0^t (T_s)'(v)(\Phi(u) - \Phi(v))\cdot\frac{1}{k} T_k (T_s(u) - T_s(v))\,d\beta\,dx\right|^2=\\
&\lim_{\sigma\rightarrow 0}\frac{1}{k}\mathbb{E}\int_D \int_0^t\left|[\Phi(u) - \Phi(v)](T_s^{\sigma})'(v)T_k(T_s^{\sigma}(u)-T^{\sigma}_s(v))-(T_s)'(v)T_k(T_s(u) - T_s(v))\right|^2\,dxdt\\
&=0,
\end{align*}
thus
\begin{align*}
&\lim_{\sigma\rightarrow 0} I_{4,2}^{\sigma,k,s}=\frac{1}{k}\int_D\int_0^t[\Phi(u) - \Phi(v)](T_s)'(v)\frac{1}{k}T_k(T_s(u) - T_s(v))\,d\beta \,dx
\end{align*}
in $L^2(\Omega)$ and, passing to a not relabeled subsequence if necessary, also a.s. in $\Omega$.
and similarly we obtain
\begin{align}\label{201215_01}
&\lim_{k\rightarrow 0}\int_D\int_0^t[\Phi(u) - \Phi(v)](T_s)'(v)\frac{1}{k}T_k(T_s(u) - T_s(v))\,d\beta \,dx\nonumber\\
&=\int_D\int_0^t[\Phi(u) - \Phi(v)](T_s)'(v)\operatorname{sign}(T_s(u) - T_s(v))\,d\beta \,dx
\end{align}
in $L^2(\Omega)$ and, passing to a not relabeled subsequence if necessary, also a.s. in $\Omega$. Passing to the limit with $s\rightarrow\infty$ in \eqref{201215_01}, we get
\begin{align}\label{201211_04}
\lim_{s\rightarrow\infty}\lim_{k\rightarrow 0}\lim_{\sigma\rightarrow 0}I_4^{\sigma,k,s}=\int_D\int_0^t[\Phi(u) - \Phi(v)]\operatorname{sign}(u -v)\,d\beta \,dx
\end{align}
for all $t\in[0,T]$, a.s. in $\Omega$. 
Now, we write
\begin{align*}
I_5^{\sigma,k,s}=&\frac{1}{2}\int_D \int_0^t (\Phi^2(u)(T_s^{\sigma})''(u) -\Phi^2(v)(T_s^{\sigma})''(v)) 
\cdot \frac{1}{k} T_k (T_s^{\sigma}(u) - T_s^{\sigma}(v)) \,dr \,dx \notag \\
&+\frac{1}{2}\int_D \int_0^t \Phi^2(v)((T_s^{\sigma})''(u) -(T_s^{\sigma})''(v)) 
\cdot \frac{1}{k} T_k (T_s^{\sigma}(u) - T_s^{\sigma}(v)) \,dr \,dx\notag\\
&\frac{1}{2}\int_D \int_0^t (T_s^{\sigma})''(u)(\Phi^2(u)-\Phi^2(v))\cdot \frac{1}{k} T_k (T_s^{\sigma}(u) - T_s^{\sigma}(v)) \,dr \,dx\notag\\
&:=I_{5,1}^{\sigma,k,s}+I_{5,2}^{\sigma,k,s}
\end{align*}
Using exactly the same arguments as in \cite{NSAZ}, p.25 for the expression $I_3^2$, we get that
\begin{align}\label{201211_05}
\limsup_{\sigma\rightarrow 0}I_{5,1}^{\sigma,k,s}\leq 0.
\end{align}
In the following we show that there exists a subsequence $(s_j)_{j\in\mathbb{N}}\subset\mathbb{N}$ with $\lim_{j\rightarrow\infty} s_j=+\infty$ such that
\begin{align}\label{201211_06}
\lim_{j\rightarrow\infty}\limsup_{k\rightarrow 0} \limsup_{\sigma\rightarrow 0}I_{5,2}^{\sigma,k,s_j}=0
\end{align}
We have, for any $k\in\mathbb{N}$,
\begin{align*}
|I_{5,2}^{\sigma,k,s}|&\leq \frac{1}{2}\mathbb{E}\int_D \int_0^t |(T_s^{\sigma})''(u)||\Phi^2(u)-\Phi^2(v)|
\cdot \frac{1}{k} |T_k (T_s^{\sigma}(u) - T_s^{\sigma}(v))| \,dr \,dx\notag\\
&\leq \frac{1}{\delta}\int_{\{s<|u|<s+\delta\}}\Vert\Phi\Vert_{\infty}^2\,dr\,dx
\end{align*}
and therefore the assertion follows from Lemma 6 in \cite{DBFMHR}.
Now we turn our attention towards $I^{\sigma,k,s}_6\leq I^{\sigma,k,s}_{6,1}+I^{\sigma,k,s}_{6,2}$, where
\begin{align*}
I^{\sigma,k,s}_{6,1}=\frac{1}{k}\int_D\int_0^t\Phi^2(u)((T_s)'(u)-(T_s)'(v))^2\chi_{\{|T_s(u)-T_s(v)|\leq k\}}\, dr\,dx,
\end{align*}
\begin{align*}
I^{\sigma,k,s}_{6,2}=\frac{1}{k}\int_D\int_0^t((T_s)'(u))^2(\Phi(u)-\Phi(v))^2\chi_{\{|T_s(u)-T_s(v)|\leq k\}}\, dr\,dx.
\end{align*}
With the same arguments as in \cite{NSAZ}, p.26 for $I_5$ and thanks to the boundedness of $\Phi$, we show that there exists a subsequence $(s_j)_{j\in\mathbb{N}}\subset\mathbb{N}$ with $\lim_{j\rightarrow\infty} s_j=+\infty$ such that
\begin{align}\label{201215_02}
\limsup_{j\rightarrow 0}\limsup_{k\rightarrow 0}\limsup_{\sigma\rightarrow 0}I^{\sigma,k,s}_{6,1}&\leq 0.
\end{align}
Moreover, 
\begin{align*}
\lim_{\sigma\rightarrow 0}I_{6,2}^{k,s}&=\frac{1}{k}\int_D\int_0^t ((T_s)'(v))^2(\Phi(u)-\Phi(v))^2\chi_{\{|T_s(u)-T_s(v)|\leq k\}}\, dr\,dx\nonumber\\
&\leq \frac{1}{k}\int_D\int_0^t ((T_s)'(v))^2(\Phi(u)-\Phi(v))^2\chi_{\{|T_s(u)-T_s(v)|\leq k\}}\chi_{\{|u|\leq s\}}\chi_{\{|v|\leq s\}}\, dr\,dx\nonumber\\
&+\frac{1}{k}\int_D\int_0^t((T_s)'(v))^2(\Phi(u)-\Phi(v))^2\chi_{\{|T_s(u)-T_s(v)|\leq k\}}\chi_{\{|u|> s\}}\chi_{\{|v|\leq s\}}\, dr\,dx.
\end{align*}
It is not very hard to see that the first term on the right-hand side of
above equation vanishes for $k\rightarrow 0$ a.s. in $\Omega$. Concerning the second term, we remark that for $\omega \in \Omega$ fixed, we have
\begin{align*}
&\limsup_{k\rightarrow 0}\frac{1}{k}\int_D\int_0^t((T_s)'(v))^2(\Phi(u)-\Phi(v))^2\chi_{\{|T_s(u)-T_s(v)|\leq k\}}\chi_{\{|u|>s\}}\chi_{\{|v|\leq s\}}\, dr\,dx\\
&\leq \limsup_{k\rightarrow 0}\frac{4}{k}\int_{\{s-k\leq|v|\leq s\}}\Vert\Phi\Vert_{\infty}^2\, dr\,dx.
\end{align*}
Consequently, from from Lemma 6 in \cite{DBFMHR} it follows that there exists a subsequence $(s_j)_{j\in\mathbb{N}}\subset\mathbb{N}$ with $\lim_{j\rightarrow\infty} s_j=+\infty$ such that
\begin{align}\label{201217_01}
\limsup_{j\rightarrow 0}\limsup_{k\rightarrow 0}\limsup_{\sigma\rightarrow 0}I^{\sigma,k,s}_{6,2}&\leq 0.
\end{align}
From \eqref{201211_01} - \eqref{201217_01} it follows that
\begin{align}\label{201215_03}
\int_D |u(t) - v(t)| \,dx \leq \int_D |u_0 - v_0| \,dx +\int_D\int_0^t[\Phi(u) - \Phi(v)]\operatorname{sign}(u -v)\,d\beta \,dx
\end{align}
a.s. in $\Omega$, for all $t \in [0,T]$. Taking the expectation in \eqref{201215_03}, the assertion follows.
\end{proof}

\section{Appendix: The It\^{o} product rule}

\begin{prop}\label{itoproduct}
For $1<p<\infty$, $u_0$, $v_0\in L^2(\Omega\times D)$ $\mathcal{F}_0$-measurable let $u,v\in L^p(\Omega\times(0,T);W^{1,p}_0(D))\cap L^2(\Omega;\mathcal{C}([0,T];L^2(D)))$ satisfy
\begin{align}\label{181201_03}
u(t)=u_0+\int_0^t\Delta_p(u)\,ds+\int_0^t\Phi(u)\,d\beta,
\end{align}
\begin{align}\label{201204_01}
v(t)=v_0+\int_0^t\Delta_p(v)\,ds+\int_0^t\Phi(v)\,d\beta.
\end{align}
Then, for any $H,Z\in \mathcal{C}^2_b(\mathbb{R})$ such that $Z(0)=Z'(0)=0$
\begin{align}\label{181201_04}
&(Z((u-v)(t)),H(u(t)))_2=(Z(u_0-v_0),H(u_0))_2\nonumber\\
&+\int_0^t\langle \Delta_p(u)-\Delta_p(v),H(u)Z'(u-v)\rangle_{W^{-1,p'}(D),W^{1,p}_0(D)}\, ds\nonumber\\
&+\int_0^t\langle \Delta_p(u),H'(u)Z(u-v)\rangle_{W^{-1,p'}(D),W_0^{1,p}(D)}\, ds+\int_0^t(\Phi(u) H'(u),Z(u-v))_2\,d\beta\nonumber\\
&+\frac{1}{2}\int_0^t\int_D\Phi^2(u)H''(u)Z(u-v)\, dx\,ds\nonumber\\
&+\frac{1}{2}\int_0^t\int_D(\Phi(u)-\Phi(v))^2Z''(u-v)H(u)\,ds\nonumber\\
&+\int_0^t\langle \Phi(u)-\Phi(v),Z'(u-v)H(u)\rangle\,d\beta\nonumber\\
&+\int_0^t(\Phi(u)-\Phi(v))Z'(u-v)\Phi(u)H'(u)\,ds
\end{align}
for all $t\in [0,T]$, a.s. in $\Omega$.
\end{prop}
\begin{proof}
We fix $t\in [0,T]$. Since $u$, $v$ satisfy \eqref{201204_01} and \eqref{181201_03}, it follows that
\begin{align}\label{201204_02}
(u-v)(t)=u_0-v_0+\int_0^t \Delta_p(u)-\Delta_p(v)\,ds +\int_0^t (\Phi(u)-\Phi(v))\,d\beta
\end{align}
holds in $L^2(D)$, a.s. in $\Omega$. For $n\in\mathbb{N}$, we use the following classical regularization procedure (see, e.g., \cite{DFEP}):\\
We choose a sequence of operators $(\Pi_n)$, 
\[\Pi_n:W^{-1,p'}(D)+L^1(D)\rightarrow W^{1,p}_0(D)\cap L^{\infty}(D),~n\in\mathbb{N}\]
such that 
\begin{itemize}
\item[$i.)$] $\Pi_n(v)\in W^{1,p}_0(D)\cap C^{\infty}(\overline{D})$ for all $v\in W^{-1,p'}(D)+L^1(D)$ and all $n\in\mathbb{N}$ 
\item[$ii.)$] For any $n\in\mathbb{N}$ and any Banach space 
\[F\in \{W^{1,p}_0(D), L^2(D), L^1(D), W^{-1,p'}(D), W^{-1,p'}(D)+L^1(D)\}.\]
$\Pi_n : F \to F$ is a bounded linear operator such that 
$\lim_{n \to \infty}{\Pi_n}_{|F}=I_F$ pointwise in $F$, where $I_F$ is the identity on $F$.
\end{itemize}
Now, we set $\Phi_{u,n}:=\Pi_n(\Phi(u))$, $\Phi_{v,n}:=\Pi_n(\Phi(v))$, $u^{n}_0:=\Pi_n(u_0)$,  $v^{n}_0:=\Pi_n(v_0)$, $u_{n}:=\Pi_n(u)$, $v_{n}:=\Pi_n(v)$, $U_n:=\Pi_n(\Delta_p(u))$, $V_n:=\Pi_n(\Delta_p(v))$. Applying $\Pi_n$ on both sides of \eqref{201204_02} yields
\begin{align}\label{201204_03}
(u_{n}-v_{n})(t)=u^{n}_0-v^{n}_0+\int_0^t U_n-V_{n}\,ds+\int_0^t (\Phi_{u,n}-\Phi_{v,n})\,d\beta
\end{align}
and applying $\Pi_n$ on both sides of \eqref{181201_03} yields
\begin{align}\label{201204_04}
u_{n}(t)=u_0^{n}+\int_0^t U_{n}\,ds+\int_0^t \Phi_{u,n} \,d\beta
\end{align}
in $W^{1,p}_0(D)\cap L^2(D)\cap \mathcal{C}^{\infty}(\overline{D})$ a.s. in $\Omega$.
The pointwise It\^{o} formula in \eqref{201204_03} and \eqref{201204_04} leads to
\begin{align}\label{201204_05}
&Z(u_{n}-v_n)(t)=Z(u^{n}_0-v_0^{n})\nonumber\\
&+\int_0^t (U_n-V_{n})Z'(u_n-v_n)\,ds\int_0^t (\Phi_{u,n}-\Phi_{v,n})Z'(u_n-v_n)\,d\beta\nonumber\\
&+\frac{1}{2}\int_0^t(\Phi_{u,n}-\Phi_{v,n})^2 Z''(u_n-v_n)\,ds
\end{align}
and
\begin{align}\label{201204_005}
H(u_{n})(t)=H(u^{n}_0)+\int_0^t U_{n} H'(u_{n})\,ds
+\int_0^t\Phi_n H'(u_n)\,d\beta+\frac{1}{2}\int_0^t\Phi_{u,n}^2H''(u_{n})\, ds
\end{align}
in $D$, a.s. in $\Omega$. From \eqref{201204_05}, \eqref{201204_005} and the  product rule for It\^{o} processes, which is just an easy application of the classic two-dimensional It\^{o} formula (see, e.g., \cite{PB}, Proposition 8.1, p. 218), applied pointwise in $t$ for fixed $x\in D$ it follows that
 \begin{align}\label{201204_06}
&Z(u_{n}-v_{n})(t)H(u_{n})(t)=Z(u^{n}_0-v^{n}_0)H(u^{n}_0)\nonumber\\
 &+\int_0^t (U_{n}-V_n)Z'(u_n-v_n)H(u_n)\,ds+\int_0^t U_n H'(u_n)Z(u_n-v_n)\,ds\nonumber\\
 &+\int_0^t \Phi_{u,n} H'(u_n)Z(u_n-v_n)\,d\beta\nonumber\\
&+\frac{1}{2}\int_0^t\Phi_{u,n}^2H''(u_n)Z(u_n-v_n)\,ds+\int_0^t (\Phi_{u,n}-\Phi_{v,n})Z'(u_n-v_n)H(u_n)\,d\beta\nonumber\\
&+\frac{1}{2}\int_0^t(\Phi_{u,n}-\Phi_{v,n})^2 Z''(u_n-v_n)H(u_n)\,ds\nonumber\\
&+\int_0^t(\Phi_{u,n}-\Phi_{v,n})Z'(u_n-v_n)\Phi_{u,n}H'(u_n)\,ds
\end{align}
in $D$, a.s. in $\Omega$. Integration over $D$ in \eqref{201204_06} yields
\begin{align*}%\label{181212_04}
I_1=I_2+I_3+I_4+I_5+I_6+I_7+I_8+I_9,
\end{align*}
where
\begin{align*}
I_1&=(Z((u_n-v_n)(t)),H((u_n)(t))_2, \nonumber\\
I_2&=(Z(u^{n}_0-v^{n}_0),H(u^n_0))_2, \nonumber\\
I_3&=\int_0^t\int_D(U_{n}-V_n)Z'(u_n-v_n)H(u_n)\,dx\,ds, \nonumber\\
I_4&=\int_0^t\int_D U_n H'(u_n)Z(u_n-v_n)\,dx\,ds, \nonumber\\
I_5&=\int_0^t (\Phi_{u,n} H'(u_n),Z(u_n-v_n))_2\,d\beta, \nonumber\\
I_6&=\frac{1}{2}\int_0^t\int_D\Phi_{u,n}^2 H''(u_n)Z(u_n-v_n)\,dx\,ds\\
I_7&=\int_0^t ((\Phi_{u,n}-\Phi_{v,n})Z'(u_n-v_n),H(u_n))_2\,d\beta\\
I_8&=\frac{1}{2}\int_0^t\int_D(\Phi_{u,n}-\Phi_{v,n})^2 Z''(u_n-v_n)H(u_n)\,dx\,ds\\
I_9&=\int_0^t((\Phi_{u,n}-\Phi_{v,n})Z'(u_n-v_n),\Phi_{u,n}H'(u_n))_2\,ds
\end{align*}
a.s. in $\Omega$. Repeating the arguments from \cite{NSAZ}, proof of Proposition 9.1, we show that, passing to a not relabeled subsequence if necessary,
\begin{align}\label{181212_05}
\lim_{n\rightarrow\infty}I_1=(Z((u-v)(t)),H'(u(t))_2,
\end{align}
\begin{align}\label{181212_06}
\lim_{n\rightarrow\infty}I_2=(Z(u_0-v_0),H'(u_0))_2,
\end{align}
\begin{align}\label{181218_03}
\lim_{n\rightarrow\infty}I_3=\int_0^t \langle \Delta_p(u)-\Delta_p(v) ,Z'(u-v)H(u)\rangle_{W^{-1,p'}(D),W_0^{1,p}(D)}\,ds,
\end{align}
\begin{align}\label{181212_12}
\lim_{n\rightarrow\infty}I_4=\int_0^t\langle \Delta_p(u),H'(u)Z(u-v)\rangle_{W^{-1,p'}(D),W_0^{1,p}(D)}\, ds,
\end{align}
\begin{align}\label{181212_13}
\lim_{n\rightarrow\infty} I_5=\int_0^t\int_D\Phi(u) H'(u)Z(u-v)\,dx\,d\beta,
\end{align}
\begin{align}\label{181213_01}
\lim_{n\rightarrow\infty}I_6=\frac{1}{2}\int_0^t\int_D\Phi(u)^2 H''(u)Z(u-v)\,dx\,ds
\end{align}
 a.s. in $\Omega$. Since $(\Pi_n)_n$ is a sequence of linear operators on $L^2(D)$ converging pointwise to the identity for $n\rightarrow\infty$, from the Uniform Boundedness Principle  and Lebesgues dominated convergence theorem it follows that $\Phi_{u,n}\rightarrow \Phi(u)$ and $\Phi_{v,n}\rightarrow\Phi(v)$ in $L^2(0,T;L^2(D))$ and in $L^2(\Omega;L^2(0,T;L^2(D))$ for $n\rightarrow\infty$. Using the It\^{o} isometry and passing to a not relabeled subsequence if necessary, it follows that
\begin{align}\label{201204_07}
\lim_{n\rightarrow\infty} I_7=\int_{0}^t(\Phi(u)-\Phi(v))Z'(u-v),H(u))_2\,d\beta,
\end{align}
\begin{align}\label{201204_08}
\lim_{n\rightarrow\infty} I_8=\frac{1}{2}\int_0^t\int_D(\Phi(u)-\Phi(v))^2 Z''(u-v)H(u)\,dx\,ds,
\end{align}
\begin{align}\label{201204_09}
\lim_{n\rightarrow\infty} I_9=\int_0^t(\Phi(u)-\Phi(v)Z'(u-v),\Phi(u)H'(u))_2\,ds.
\end{align}
Now, the assertion follows from \eqref{181212_05}-\eqref{201204_09}.
\end{proof}
\begin{cor}\label{cor1}
Proposition \ref{itoproduct} still holds true for $H,Z\in W^{2,\infty}(\mathbb{R})$ such that $H''$ and $Z''$ are piecewise continuous.
\end{cor}
\begin{proof}
There exists sequence $(H_{\delta})_{\delta>0}$, $(Z_{\delta})_{\delta>0}\subset \mathcal{C}^2_b(\mathbb{R})$ such that $\Vert H_{\delta}\Vert_{\infty}\leq \Vert H\Vert_{\infty}$, $\Vert H'_{\delta}\Vert_{\infty}\leq \Vert H'\Vert_{\infty}$, $\Vert H''_{\delta}\Vert_{\infty}\leq \Vert H''\Vert_{\infty}$ for all $\delta>0$ and $H_{\delta}\rightarrow H$, $H'_{\delta}\rightarrow H'$ uniformly on compact subsets, $H_{\delta}''\rightarrow H''$ pointwise in $\mathbb{R}$ for $\delta\rightarrow 0$ and the same results hold true for $(Z_{\delta})_{\delta>0}$. With these convergence results we are able to pass to the limit with $\delta\rightarrow 0$ in \eqref{181201_04}.
\end{proof}

\end{document}